\documentclass[a4paper,11pt]{article}

\usepackage{amsmath,amsthm,amssymb,mathrsfs,latexsym,amsfonts}
\usepackage{graphicx,psfrag,epsfig}
\usepackage[english]{babel}
\usepackage[latin1]{inputenc}

\newtheorem{theoreme}{Theorem}[section]

\newtheorem{lemme}[theoreme]{Lemma}
\newtheorem{proposition}[theoreme]{Proposition}

\newcounter{paraga}[section]
\renewcommand{\theparaga}{{\bf\arabic{paraga}.}}
\newcommand{\paraga}{\medskip \addtocounter{paraga}{1} 
\noindent{\theparaga\ } }

\begin{document}

\bibliographystyle{amsalpha}

\def\MP{\,{<\hspace{-.5em}\cdot}\,}
\def\SP{\,{>\hspace{-.3em}\cdot}\,}
\def\PM{\,{\cdot\hspace{-.3em}<}\,}
\def\PS{\,{\cdot\hspace{-.3em}>}\,}
\def\EP{\,{=\hspace{-.2em}\cdot}\,}
\def\PP{\,{+\hspace{-.1em}\cdot}\,}
\def\PE{\,{\cdot\hspace{-.2em}=}\,}
\def\N{\mathbb N}
\def\C{\mathbb C}
\def\Q{\mathbb Q}
\def\R{\mathbb R}
\def\T{\mathbb T}
\def\A{\mathbb A}
\def\Z{\mathbb Z}
\def\demi{\frac{1}{2}}
\def\eps{\varepsilon}
\def\pfrac#1#2{{\textstyle{#1\over#2}}}
\def\abs#1{\vert #1\vert}

\begin{titlepage}
\author{Abed Bounemoura~\footnote{Laboratoire de Mathématiques d'Orsay et Institut de Mathématiques de Jussieu} {} and 
Jean-Pierre Marco~\footnote{Institut de Mathématiques de Jussieu}}
\title{\LARGE{\textbf{Improved exponential stability for near-integrable quasi-convex Hamiltonians}}}
\end{titlepage}

\maketitle

\begin{abstract}
In this article, we improve previous results on exponential stability for analytic and Gevrey perturbations of quasi-convex integrable Hamiltonian systems. In particular, this provides a sharper upper bound on the speed of Arnold diffusion which we believe to be optimal.
\end{abstract}
  
\section{Introduction}

This paper deals with some stability properties of near-integrable Hamiltonian systems of the form
\begin{equation*}
\begin{cases} 
H(\theta,I)=h(I)+f(\theta,I) \\
|f| < \varepsilon <\!\!<1
\end{cases}
\end{equation*}
where $(\theta,I) \in \T^n \times \R^n$ are action-angle coordinates for the integrable part $h$ and $f$ is a small perturbation, of size $\varepsilon$ in some suitable topology defined by a norm $|\,.\,|$. If the system is analytic and under a suitable quantitative transversality condition called {\it steepness}, Nekhoroshev (\cite{Nek77}, \cite{Nek79}) proved that the action variables $I(t)$ are stable for an exponentially long time, in the sense that for $\varepsilon$ sufficiently small, one has
\[ |I(t)-I_0| \leq c_1\varepsilon^b, \quad |t|\leq c_2\exp(c_3\varepsilon^{-a}) \] 
for any initial action $I_0$. The constants $c_1,c_2,c_3,a$ and $b$ depend only $h$, one calls $a$ and $b$ the {\it stability exponents} and among them the value of $a$ is of course the most important one, as it specifies the time-scale of stability. Nekhoroshev's estimates complement the well-known KAM theory (\cite{Kol54}, see also \cite{Pos01} for a nice survey) which gives, under some mild non-degeneracy condition on $h$ and for $\varepsilon$ small enough, the existence of a constant $c$ such that
 \[ |I(t)-I_0| \leq c\sqrt\varepsilon \]
for any time $t\in\R$, but only for a strict subset  (of large relative measure) of the set of initial conditions. Of course KAM theory gives much more information, these stable solutions are in fact quasi-periodic and $\sqrt \varepsilon$-close to the corresponding unperturbed solutions. In particular, for $n=2$ and in the case when $h$ is iso-energetically non-degenerate, this stability property even holds for all solutions. On the contrary, for $n\geq 3$, following Arnold (\cite{Arn64}) one can find examples of near-integrable Hamiltonian systems with a solution satisfying
\[ |I(\tau)-I_0| \geq 1 \]
for some time $\tau=\tau(\varepsilon)>0$, no matter how small the perturbation is. Such instability is commonly referred to as Arnold diffusion. Obviously Nekhoroshev's estimates give a lower bound on the diffusion time $T$ (or equivalently an upper bound on the diffusion speed) which is exponentially large (or exponentially small when referring to the rate of diffusion). 

This paper is concerned with the precise time-scale at which stability breaks down and instability takes place, by which we mean the precise value of the exponent $a$, in the special case where the unperturbed Hamiltonian $h$ is quasi-convex (that is convex when restricted to its level subsets). 

The quasi-convex case, which is of both practical and theoretical interest, has been widely studied in Nekhoroshev theory, essentially for two reasons. First the proof is much easier in this situation and a more refined result is available: the stability exponents may be chosen as
\[ a=b=(2n)^{-1}. \] 
These facts are best illustrated by the striking proof given by Lochak (\cite{Loc92}, see also \cite{LN92} and \cite{LNN94}), though they are also accessible via a more traditional approach as was shown by Pöschel (\cite{Pos93}). Note also that these values have been generalized by Niederman in the steep case (\cite{Nie04}) using both ideas of Lochak and Pöschel. Yet there is another reason for which the quasi-convex case is interesting, which is the so-called {\it stabilization by resonances}: if a solution starts close to a resonance of multiplicity $m$, $m<n$, then it posses better stability properties, described by the ``local" exponents
\[ a_m=b_m=(2(n-m))^{-1}. \]  
This is a quite surprising fact, as it shows that even though resonances are the main cause of diffusion, at the same time they improve finite time stability. However, this property certainly does not hold without some convexity assumption (in the steep case for instance).

The optimality of the exponent $a$, in connection with the maximal speed of Arnold diffusion, has been first studied by Bessi who introduced powerful variational methods to revisit Arnold's example and estimate the speed of diffusion. In \cite{Bes96} and \cite{Bes97}, he proved that the latter is of order $\exp\big(\varepsilon^{-\frac{1}{2}}\big)$ for $n=3$ and $\exp\big(\varepsilon^{-\frac{1}{4}}\big)$ for $n=4$. Moreover, in Bessi's example the solution passes close to a double resonance, and so the speed is the highest possible in this case, in view of the values of the local exponent $a_2$ for $n=3$ and $n=4$. Recently, using similar variational arguments, this result has been generalized to an arbitrary number of degrees of freedom $n$ by Ke Zhang (\cite{KZ09}), namely he constructed a special orbit passing close to a double resonance for which the speed of diffusion is estimated by $\exp\big(\varepsilon^{-\frac{1}{2(n-2)}}\big)$.  

Another approach has been proposed by Marco-Sauzin (\cite{MS02}) and Lochak-Marco (\cite{LM05}), following novel ideas of Herman. In \cite{MS02} the authors showed that Nekhoroshev's estimates extend to perturbations of quasi-convex Hamiltonians which are $\alpha$-Gevrey regular, $\alpha\geq 1$, with the exponents
\[ a=(2\alpha n)^{-1}, \quad b=(2n)^{-1} \]
and local exponents
\[ a_m=(2\alpha(n-m))^{-1}, \quad b_m=(2(n-m))^{-1}. \]  
Note that $1$-Gevrey functions are exactly analytic functions, and basically when $\alpha$ ranges from one to infinity $\alpha$-Gevrey functions interpolate between analytic and $C^\infty$ functions. Therefore this result generalizes the estimates in the analytic case. Using a geometric mechanism different and more precise than Arnold's one, in \cite{MS02} the authors constructed a drifting orbit with speed of order $\exp\big(\varepsilon^{-\frac{1}{2\alpha(n-2)}}\big)$ in the non-analytic case, that is when $\alpha>1$. Adding some more technical ideas, it was shown in~\cite{LM05} that the example also works in the analytic case but the speed was estimated as $\exp\big(\varepsilon^{-\frac{1}{2(n-3)}}\big)$, which is only close to optimal
(however refinements are certainly possible to reach the value $(2(n-2))^{-1}$ in this class of examples).

Therefore, if the unperturbed Hamiltonian is quasi-convex, the best exponent of stability $a$ up to now satisfy
\[ (2n)^{-1} \leq a < (2(n-2))^{-1} \]
in the analytic case and more generally 
\[ (2\alpha n)^{-1} \leq a < (2\alpha(n-2))^{-1} \]
in the Gevrey case. The goal of this paper is to improve the lower bound both in the analytic or Gevrey case, so as to have
\[ (2(n-1))^{-1}-\delta \leq a < (2(n-2))^{-1} \]  
and 
\[ (2\alpha(n-1))^{-1}-\delta \leq a < (2\alpha(n-2))^{-1} \] 
for $\delta>0$ but arbitrarily small (see Theorem~\ref{exposantNek} and Theorem~\ref{exposantNekG} in the next section). We believe that this bound is optimal, in the sense that one could reach the value $(2(n-1))$ in Arnold diffusion, using a significantly different mechanism of instability.

\section{Main results}

\paraga In order to state our main results, let us now describe our setting more precisely, beginning with the analytic case. Let $B=B(0,R)$ be the open ball of $\R^n$ of radius $R>0$, with respect to the supremum norm, centered at the origin. Given $s>0$, we let $\mathcal{A}_s(\mathcal{D})$ the space of bounded real-analytic functions on $\mathcal{D}=\T^n \times B$ which extend as holomorphic functions on the complex domain
\[ \mathcal{D}_{s}=\{(\theta,I)\in(\C^n/\Z^n)\times \C^{n} \; | \; |\mathcal{I}(\theta)|<s,\;d(I,B)<s\}, \] 
and which are continuous on the closure of $\mathcal{D}_{s}$. Here we denoted by $\mathcal{I}(\theta)$  the imaginary part of $\theta$, by $|\,.\,|$ the supremum norm on $\C^n$ and by $d$ the associated distance on $\C^n$. It is well-known that $\mathcal{A}_s(\mathcal{D})$ is a Banach space with its usual supremum norm $|\,.\,|_s$, where
\[ |f|_s=\sup_{z\in\mathcal{D}_{s}}|f(z)|, \quad f\in\mathcal{A}_s(\mathcal{D}). \]
In the following, we shall denote by
\[ B_s=\{I\in\R^n \; | \; d(I,B)<s\} \]
the real part of our domain $\mathcal{D}_s$ in action space. The geometric parameters $n,R,s$ are assumed to be chosen  once and for all in the following.

We now introduce the parameters related to the choice of the system. The integrable part $h:B_s\to\R$ will be assumed to be strictly quasi-convex: the gradient map $\nabla h$ does not vanish and there exists a positive number $m$ such that 
\begin{equation}\label{mconvex} \tag{$QC(m)$}
\nabla^2 h(I)v.v \geq m|v|^2 
\end{equation} 
holds for any $I\in B_s$ and any $v$ orthogonal to $\nabla h(I)$ (with respect to the Euclidean scalar product). Moreover, the derivatives up to order 3 of $h$ on $B_s$ are assumed to be bounded: there exist $M>0$ such that for all $I\in B_s$, one has
\begin{equation}\label{bounded} \tag{$B(M)$}
|\partial ^k h(I)|\leq M, \quad 1\leq |k|\leq 3.
\end{equation}
Therefore we will consider
systems of the form
\begin{equation}\label{Hamexpo}
\begin{cases} \tag{$C(M,m,\varepsilon)$}
H(\theta,I)=h(I)+f(\theta,I), \\
h\in \mathcal{A}_s(\mathcal{D}),\ f \in \mathcal{A}_s(\mathcal{D}),\\
h \;  \text{satisfies} \; (QC(m))  \;\text{and}\; (B(M)),\\
|f|_{s}<\varepsilon.
\end{cases}
\end{equation}

Note that we get rid of the geometric parameters in the notation.
 In the following we will call {\em stable constant} (in the analytic case)
any positive constant $c$ which depends on the whole set of parameters, that is 
$n,R,s,M,m$, together with a parameter $\delta$ or $\rho$ to be defined below, but not on a particular choice of $H$ satisfying condition $(C(M,m,\varepsilon))$.

\paraga The main result of the paper is the following.

\begin{theoreme}\label{exposantNek}
Consider a real number $\delta$ satisfying
\[ 0<\delta \leq (2n(n-1))^{-1}. \]
Then there exist stable constants $c_1$, $c_2$, $c_3$ and $\varepsilon_0$ such that if  $0\leq\varepsilon \leq \varepsilon_0$, and if $H$ satisfies $(C(M,m,\varepsilon))$, the following estimates
\[ |I(t)-I_0| \leq c_1\varepsilon^{\delta(n-1)}, \quad |t| \leq c_2 \exp\left(c_3\varepsilon^{-\frac{1}{2(n-1)}+\delta}\right) \]
hold true for every initial action $I_0 \in B(0,R/2)$. 

Moreover, consider a real number $\rho$ satisfying $0<\rho<R/2$. Then there exist stable constants $c_4$, $c_5$ and $\tilde{\varepsilon}_0$ such that if $0\leq\varepsilon \leq \tilde{\varepsilon}_0$, then
\[ |I(t)-I_0| \leq \rho, \quad |t|\leq c_4 \exp\big(c_5\varepsilon^{-\frac{1}{2(n-1)}}\big) \]
for every $I_0 \in B(0,R/2)$.
\end{theoreme}

Choosing our constant $\delta$ arbitrarily close to zero, our result ensures stability for a time-scale which is arbitrarily close to $\exp\big(\varepsilon^{-\frac{1}{2(n-1)}}\big)$, therefore we improve the previous results of stability obtained independently by Lochak-Neishtadt (\cite{Loc92} and \cite{LN92}) and Pöschel (\cite{Pos93}), which were believed to be optimal.  

In fact in the extreme case where $\delta=(2n(n-1))^{-1}$, which in our situation gives the worst stability time (but of course the best radius of confinement), our result reads
\[ |I(t)-I_0| \leq c_1\varepsilon^{\frac{1}{2n}}, \quad |t|\leq c_2 \exp\big(c_3\varepsilon^{-\frac{1}{2n}}\big) \] 
and we recover the previous result of stability. Hence, when our parameter $\delta$ ranges from $(2n(n-1))^{-1}$ to zero, our Theorem ``interpolates" between previous results of stability and what should be the optimal stability. 

Indeed, in the other extreme case which corresponds to the second part of our Theorem, our result does not give stability since the radius of confinement can be arbitrarily small but no longer tends to $0$ with $\varepsilon$. We believe that this is not an artefact of the method and that instability should occur at this precise time-scale, at a time of order $\exp\big(\varepsilon^{-\frac{1}{2(n-1)}}\big)$. We plan to construct an example with an unstable orbit which has a drift of order one during such an interval of time. This necessitates to use a more refined instability mechanism in the neighbourhood of double resonances, a topic which is also crucial in connection with the problem of genericity of Arnold diffusion.

\paraga Our result also holds if the Hamiltonian is only Gevrey regular. Let us recall that given $\alpha \geq 1$ and $L>0$, a function $H\in C^{\infty}(\mathcal{D})$ is $(\alpha,L)$-Gevrey if, using the standard multi-index notation, we have
\[ |H|_{\alpha,L}=\sum_{l\in \N^{2n}}L^{|l|\alpha}(l!)^{-\alpha}|\partial^l H|_\mathcal{D} < \infty \]
where $|\,.\,|_\mathcal{D}$ is the usual supremum norm for functions on $\mathcal{D}$. The space of such functions, with the above norm, is a Banach space that we denote by $G^{\alpha,L}(\mathcal{D})$. Analytic functions are a particular case of Gevrey functions, as one can check that $G^{1,L}(\mathcal{D})=\mathcal{A}_{L}(\mathcal{D})$.

Let us introduce the main condition on the Hamiltonian systems in the Gevrey case
\begin{equation}\label{HamexpoG}
\begin{cases} \tag{$C(\alpha,L,M,m,\varepsilon)$}
H(\theta,I)=h(I)+f(\theta,I), \\
h\in G^{\alpha,L}(\mathcal{D}),\ f \in G^{\alpha,L}(\mathcal{D}),\\
h \;  \text{satisfies} \; (QC(m))  \;\text{and}\; (B(M)),\\
|f|_{\alpha,L}<\varepsilon.
\end{cases}
\end{equation}

We now call {\em stable constant} (in the Gevrey case)  any positive constant $c$ which depends on the whole set of parameters, that is $\alpha,L,n,R,s,M,m$, together with a parameter $\delta$ or $\rho$ which will be defined below, but not on a particular choice of $H$ satisfying condition $(C(\alpha,L,M,m,\varepsilon))$.

\paraga Our second result is the following Gevrey version of Theorem \ref{exposantNek}.

\begin{theoreme}\label{exposantNekG}
Consider a real number $\delta$ satisfying
\[ 0<\delta \leq (2\alpha n(n-1))^{-1}. \]
Then there exist stable constants $c_1'$, $c_2'$, $c_3'$ and $\varepsilon_0'$ such that if $0\leq\varepsilon \leq \varepsilon_0'$, and if $H$ satisfies $C(\alpha,L,M,m,\varepsilon)$, the following estimates
\[ |I(t)-I_0| \leq c_1'\varepsilon^{\frac{2\delta}{5(n-1)}}, \quad |t| \leq c_2' \exp\left(c_3'\varepsilon^{-\frac{1}{2\alpha(n-1)}+\delta}\right), \]
hold true for every initial action $I_0 \in B(0,R/2)$. 

Moreover, consider a real number $\rho$ satisfying $0<\rho<R/2$. Then there exist stable constants $c_4'$, $c_5'$ and $\tilde{\varepsilon}_0'$ such that if $0\leq\varepsilon \leq \tilde{\varepsilon}_0'$, then
\[ |I(t)-I_0| \leq \rho, \quad |t|\leq c_4' \exp\big(c_5'\varepsilon^{-\frac{1}{2\alpha(n-1)}}\big) \]
for every $I_0 \in B(0,R/2)$.
\end{theoreme}

The same remarks as above apply in the Gevrey case. In particular $\delta$ can be chosen arbitrarily close to zero and our result ensures stability for an interval of time which is arbitrarily close to $\exp\big(\varepsilon^{-\frac{1}{2\alpha(n-1)}}\big)$. However, our radius of stability is worse than in the analytic case, so we do not fully recover the result obtained in \cite{MS02}, but of course the time of stability is the most important issue.

\paraga  To avoid cumbersome expressions in the following, when there is no risk of confusion
we will replace the stable constants with a dot.  More precisely, an assertion of the form ``there exists a stable constant $c$ such
that $f < c\,g$'' will be simply replaced with ``$f\MP g$'', when the context is clear. {\em Such modifications will only concern assertions
stating the existence of stable constants, and dealing with equalities or (strict or large) inequalities of well-defined functions}.

\paraga In the rest of the paper, as usual and without loss of generality, we will only consider solutions starting at time $t=0$ and evolving in positive time, and initial conditions will be denoted by $(I_0,\theta_0)=(I(0),\theta(0))$.

\paraga Finally in the rest of this text all the norms will be denoted by $|\,.\,|$, this will always be the supremum norm for vectors and the induced norm for matrices, except for vectors in $\Z^n$ for which $|\,.\,|$ will stand for the $\ell^1$--norm. 

\section{The analytic case}

In this section, the geometric constants $n,s,R$ together with the constants $m$ and $M$ are fixed once and for all.

The proof of Theorem~\ref{exposantNek} relies on two elementary facts. The first one, which we recalled in the introduction, is the stabilizing effect of resonances: on account of quasi-convexity, solutions of any Hamiltonian satisfying~(\ref{Hamexpo}) and starting sufficiently close to a resonance are stable for a longer interval of time. Of course this concerns only some special solutions, but we need to consider all of them. Our second remark is that to deal with  the remaining ones, one can take advantage of the geometry of resonances in the integrable system to obtain a confinement result. More precisely, using the iso-energetic non-degeneracy implied by our quasi-convexity assumption, we will show that solutions that avoid all resonances are necessarily stable for all time. 

\bigskip

\paraga Let us begin by making the first point explicit, using Pöschel's approach of Nekhoroshev's theory. We shall denote by $\Omega=\nabla h(B)$ the space of frequencies. Let $\Lambda$ be a sub-module of $\Z^n$ of rank $r$, with $r\in\{1,\dots,n\}$. The resonant space associated with $\Lambda$ is defined by
\[ R_{\Lambda}=\{\omega\in \Omega \; | \; k.\omega=0, \; \forall k\in\Lambda\}. \]
Given a real number $K \geq 1$, we will say that $\Lambda$ is a $K$-sub-module if it admits a $\Z$-basis $\{k^1, \dots, k^r\}$ satisfying $|k^i|\leq K$ for $i\in\{1,\dots,r\}$, where 
\[|k^i|=|k^i_1|+\cdots+|k^i_n|.\]
As usual, it is enough to consider only maximal sub-modules, which are those that are not strictly contained in any other sub-module of the same rank. Given such a sub-module, we define its volume by
\[ |\Lambda|=\sqrt{\det {}^t\!MM} \]
where $M$ is any $n\times r$ matrix whose columns form a basis for $\Lambda$ (this is easily seen to be independent of the choice of such a matrix). The following stability Theorem is due to Pöschel. 

\begin{theoreme}[Pöschel]\label{theoremPos}
Let $\Lambda$ be a $K$-sub-module of $\Z^n$ of rank $r$, with $r\in\{0,\dots,n-1\}$. Assume that $\eps\geq 0$ and $K\geq 1$ satisfy
\begin{equation*}
\varepsilon |\Lambda|^{2} K^{2(n-r)} \MP 1
\end{equation*}
and $H$ satisfies~(\ref{Hamexpo}). Then, if $\omega(0)=\nabla h(I_0)$, for any solution $(\theta(t),I(t))$ with $I_0\in B(0,R/2)$ and $d(\omega(0),R_\Lambda) \MP \sqrt \varepsilon$ the following estimates
\[ |I(t)-I_0| \MP \left(\varepsilon |\Lambda|^{2}\right)^{\frac{1}{2(n-r)}}, \quad t \MP \exp\left(\cdot \varepsilon |\Lambda|^{2}\right)^{-\frac{1}{2(n-r)}},\]
hold true.
\end{theoreme}

As we recalled in the introduction, we use a dot in the various inequalities to abbreviate an assertion such as ``there exists a stable constant $c_i$ such that'', located at the beginning of the statement.

The previous Theorem exactly gives the content of Theorem 3 in~\cite{Pos93}, to which we refer for a possible choice of stable constants. Note that Pöschel uses a more quantitative version of quasi-convexity, namely that there exist two positive numbers $l$ and $m$ such that at least one of the inequalities
\[ |\nabla h(I).v|>l|v|, \quad \nabla^2 h(I)v.v \geq m|v|^2, \]
holds for any $I\in B_s$ and any $v\in\R^n$. It turns out that this notion is in fact equivalent to our condition~(\ref{mconvex}).

We shall only need Pöschel's result in the special case where the $K$-sub-module $\Lambda$ has rank $1$ and where the solution starts precisely \emph{on} the associated resonant manifold. So we let
\[ R_{K}=\bigcup_{\Lambda}R_{\Lambda} \]
where the union is taken over all (maximal) $K$-sub-modules of rank $1$, so it is the set of simply resonant frequencies of order $K$. Moreover, the ``volume" of a rank-one sub-module $|\Lambda|$ is nothing but the Euclidean norm of its (two) $\Z$--generators, so one gets the trivial estimate 
\[1 \leq |\Lambda|^2 \leq K^2\]
from which we deduce the following lemma. 
 
\begin{lemme}\label{lemmePos}
Let $\Lambda$ be a $K$-sub-module of $\Z^n$ of rank $1$. Assume that $\eps\geq 0$ and $K\geq 1$ satisfy
\begin{equation}\label{smalln1}
\varepsilon K^{2n} \MP 1
\end{equation}
and $H$ satisfies~(\ref{Hamexpo}). Then for any solution $(\theta(t),I(t))$ such that $I_0 \in B(0,R/2)$ and $\omega(0)\in R_K$ the following estimates
\[ |I(t)-I_0| \MP \left(\varepsilon K^{2}\right)^{\frac{1}{2(n-1)}}, \quad t \MP \exp\left(\cdot \varepsilon K^{2}\right)^{-\frac{1}{2(n-1)}},\]
hold true.
\end{lemme}  

\vskip2mm

\paraga Let us now turn to our second remark, which is a simple geometric property of the integrable system based on the iso-energetic non-degeneracy. The latter condition is known to be implied by quasi-convexity, and it can be interpreted in various ways (see (\cite{Sev06}) and references therein), but for subsequent arguments we will adopt the following form. 

\begin{lemme}\label{iso-energetic}
Suppose $h$ satisfies~(\ref{mconvex}). Then the map
\[\begin{array}{ccccc}
\Psi_{h} & : & B \times \R_*^+ & \longrightarrow & \R \times (\R^n\setminus\{0\}) \\
 & & (I,\lambda) & \longmapsto & (h(I),\lambda\omega(I)).
\end{array} \] 
is a local diffeomorphism in the neighbourhood of any point $(I_0,\lambda_0)\in B \times \R_*^+$. 

In particular, there exist stable constants $\rho_0$ and $C$ such that if $\rho<\rho_0$ then $\Psi_h$ is a diffeomorphism in restriction to the ball $B((I_0,\lambda_0),\rho)\subseteq B \times \R_*^+$, whose image contains the closed ball $\overline{B}(\Psi_h(I_0),C\rho) \subseteq \R \times \R^n$.
\end{lemme} 

The second statement is a consequence of the first one (see \cite{MS02} lemma 3.17 for quantitative estimates on the stable constants $\rho_0$ and $C$), and the first statement is well-known but we give a proof for convenience. 

\begin{proof} 
By the inverse function Theorem, it is enough to prove that $d_{(I_0,\lambda_0)}\Psi_h$ is non-singular at any point $(I_0,\lambda_0)\in  B\times \R_*^+$. Given $u\in \R$, $v\in\R^n$, we easily compute
\[ d_{(I_0,\lambda_0)}\Psi_h(v,u)=(\omega(I_0).v, u\omega(I_0)+\lambda_0\nabla^2 h(I_0)v) \]
and we need to show that this vector is non-zero if the vector $(v,u)\in\R^{n+1}$ is non-zero. If either $\omega(I_0).v \neq 0$, in which case the first component is non-zero, or $v=0$ and hence the second component is non-zero (since $u\neq 0$ and so $u\omega(I_0)\neq 0$), the statement is obvious. Otherwise, $\omega(I_0).v=0$ and $v\neq 0$, since $h$ satisfies~(\ref{mconvex}) this gives
\begin{eqnarray*}
(u\omega(I_0)+\lambda_0\nabla^2 h(I_0)v).v & = & \lambda_0\nabla^2 h(I_0)v.v \\
& \geq & \lambda_0 m|v|^2
\end{eqnarray*}
and therefore $u\omega(I_0)+\lambda_0\nabla^2 h(I_0)v\neq 0$. 
\end{proof}

\paraga We can now make one step further in the dynamical consequences of the structure of simple resonances. We consider a Hamiltonian $H$ satisfying~(\ref{Hamexpo}). We will actually focus on those simple resonances for which  the ratio of two frequencies becomes rational. The following elementary lemma will allow us to deal with this simple case.

\begin{lemme}\label{lemmeratio}
Let $I$ be a closed interval of length $l>0$ contained in $[-1,1]$. Then there exist a rational number $p/q\in I \cap \Q$ satisfying
\[ |q|+|p| < (4\sqrt 2) l^{-\frac{1}{2}}. \]
\end{lemme}

The exponent in $l^{-\frac{1}{2}}$ comes from the use of Dirichlet's Theorem on the approximation of real numbers by rational ones, but this result is not necessary as in the sequel a trivial bound of order $l^{-1}$ would be enough. 

\begin{proof}
Let us write $I=[x-l/2,x+l/2]$ for some $x\in [-1,1]$, and let $q$ be the smallest integer larger than $\sqrt{2}l^{-\frac{1}{2}}$, that is
\[ \sqrt{2}l^{-\frac{1}{2}}\leq q < \sqrt{2}l^{-\frac{1}{2}}+1. \]
By Dirichlet's Theorem there exists an integer $p\in\Z$ such that 
\[|x-p/q|<q^{-2}.\]
Since $\sqrt{2}l^{-\frac{1}{2}}\leq q$, we have $q^{-2}\leq l/2$ and so
\[|x-p/q|<l/2\]
which means that $p/q \in I$. Moreover, as $I\subset [-1,1]$, then $|p|\leq q$ and
\[ |q|+|p| \leq 2q. \]
Recall that $q < \sqrt{2}l^{-\frac{1}{2}}+1$, but as $I\subseteq [-1,1]$ we have $l\leq 2$ so that $1\leq\sqrt{2}l^{-\frac{1}{2}}$ and therefore
\[ q < (2 \sqrt{2})l^{-\frac{1}{2}}. \]
This gives
\[ |q|+|p| < (4 \sqrt{2})l^{-\frac{1}{2}} \]
which concludes the proof.
\end{proof}

The following result is our main lemma. It essentially says that a (long) drifting orbit has to cross a simple resonance, since all other orbits are stable on the interval of time over which they are defined.

\begin{lemme}\label{lemmenonresonant} 
Consider $\eps\geq 0$ and $K\geq 1$ such that 
\begin{equation}\label{smalln2}
K^{-2} \MP 1, \quad \varepsilon K^2 \MP 1.
\end{equation}
Let $H$ be a Hamiltonian satisfying (\ref{Hamexpo}), $\tau\in\R^+\cup\{+\infty\}$ and let $(\theta(t),I(t))$ be a solution defined on $[0,\tau[$ with $I_0\in B(0,R/2)$. If
\[ \omega(t) \notin R_K, \quad 0\leq t<\tau\]
then the inequality
\[ |I(t)-I_0|\MP K^{-2}, \quad 0\leq t<\tau \]
holds true.
\end{lemme}

Once again, the exponent in $K^{-2}$ comes from lemma~\ref{lemmeratio} and hence from Dirichlet's Theorem, but a bound of order $K^{-1}$ would be sufficient for the final result. Let us also add the if $\tau$ is the maximal time of existence of the solution within the initial domain $\T^n\times B$, then our stability estimate easily ensures that $\tau=+\infty$ which in turn implies stability for all time (see the proof of Theorem~\ref{exposantNek}). 

\begin{proof}
We will make conditions~(\ref{smalln2}) explicit and prove that 
\[ |I(t)-I_0|< 32\:C^{-1}K^{-2}, \quad 0\leq t<\tau, \]
when
\begin{equation*}
K^{-2} < (32)^{-1}C\rho_0, \quad \varepsilon K^2 <16,
\end{equation*}
where $\rho_0$ and $C$ are the stable constants of lemma~\ref{iso-energetic}.

We will argue by contradiction, so we assume that there exists a time $\tilde{t}$ for which
\[ |I(\tilde{t})-I_0|\geq 32\:C^{-1}K^{-2}.\]
Consider the curve
\[ \sigma(t)=(I(t),|\omega(t)|^{-1})\in B\times \R^+ \]
and let $\rho=32\:C^{-1}K^{-2}$. Then 
\[ t^*=\inf\{t\in[0,\tau[ \; | \; \sigma(t) \notin B(\sigma(0),\rho) \} \]
is well-defined as the above set contains $\tilde{t}$. Now $\rho<\rho_0$ so we can apply lemma~\ref{iso-energetic} : the restriction of $\Psi_h$ to the open ball $B(\sigma(0),\rho)$ is a diffeomorphism whose image contains the closed ball $\overline{B}(\Psi_h(\sigma(0)),32\:K^{-2})$. Considering a slightly larger ball over which $\Psi_h$ remains a diffeomorphism, this easily implies that
\[ \Psi_h(\sigma(t^*))\notin B(\Psi_h(\sigma(0)),32\:K^{-2})\]
that is
\[ \left|\left(h(I(t^*)),|\omega(t^*)|^{-1}\omega(t^*)\right)-\left(h(I_0),|\omega(0)|^{-1}\omega(0)\right)\right|\geq 32\:K^{-2}.\]
Using the preservation of energy one has 
\[ |h(I(t^*))-h(I_0)| < 2\varepsilon < 32\:K^{-2} \]
so that necessarily
\[ \left||\omega(t^*)|^{-1}\omega(t^*)-|\omega(0)|^{-1}\omega(0)\right|\geq 32\:K^{-2}.\]
Therefore there exists an index $i\in\{1,\dots,n\}$ such that 
\[\left|\frac{\omega_i(t^*)}{|\omega(t^*)|}-\frac{\omega_i(0)}{|\omega(0)|}\right|\geq 32\:K^{-2}\]
and this estimate means that the image of the interval $[0,t^*]$ under the continuous function
\[ t \longmapsto \frac{\omega_i(t)}{|\omega(t)|} \in [-1,1]\]
contains a non-trivial interval $I$ of length $l=32\:K^{-2}$. Now we can apply lemma~\ref{lemmeratio} to find a rational number $p/q \in \Q$ in reduced form and a time $t'\in[0,t^*]$ such that
\begin{equation}\label{eqrat}
\frac{\omega_i(t')}{|\omega(t')|}=\frac{p}{q}
\end{equation}
with
\begin{equation}\label{estimK}
|p|+|q| < 4\sqrt{2}(32\:K^{-2})^{-\demi}=K.
\end{equation}
But as $|\omega(t')|=|\omega_j(t')|$ for some $j\in\{1,\dots,n\}$, and replacing $p$ with $-p$ if $\omega_j(t')$ is negative, the equality~(\ref{eqrat}) can be written as
\begin{equation}\label{eqrat2}
q\omega_i(t')-p\omega_j(t')=0. 
\end{equation}
Now let us write $k'=qe_i-pe_j\in\Z$, then from~(\ref{estimK}) and~(\ref{eqrat2}) we have
\[k'.\omega(t')=0, \quad |k'|<K,\] 
and since the sub-module generated by $k'$ is maximal as $p$ and $q$ are co-primes, we have found that $\omega(t')\in R_K$. This gives the desired contradiction. 
\end{proof}

\paraga We can finally pass to the proof of Theorem~\ref{exposantNek}.

\begin{proof}[Proof of Theorem~\ref{exposantNek}]
We choose $K$ of the form
\[ K=K_0\left(\frac{\varepsilon_0}{\varepsilon}\right)^{\gamma} \]
with suitable stable constants $K_0$ and $\varepsilon_0$ so that conditions~(\ref{smalln1}) and~(\ref{smalln2}) are satisfied if
\[ 0\leq\varepsilon\leq\varepsilon_0, \quad 0 < \gamma \leq (2n)^{-1}. \]

With this threshold and these bounds on $\gamma$ both lemma~\ref{lemmePos} and lemma~\ref{lemmenonresonant} can be applied. Let $(\theta_0,I_0)\in\T^n\times B(0,R/2)$ and $T$ the maximal time of existence within $\T^n\times B(0,R)$ of the solution $(\theta(t),I(t))$ starting at $(\theta_0,I_0)$. We have to distinguish between two cases.

In the first case, we assume that $\omega(t)\in N_K$ for all $t<T$. Then we apply lemma~\ref{lemmenonresonant} with $\tau=T$ to get
\[ |I(t)-I_0|\MP\varepsilon^{2\gamma}, \quad 0 \leq t<T. \]
From this estimate we can deduce that the solution $I(t)$ belongs to some compact ball around $I_0$ which, taking $\varepsilon_0$ small enough (this is possible since $\gamma>0$), is included in $B(0,R)$. Therefore this solution is defined for all time, that is $T=+\infty$, and so the previous estimate gives
\[ |I(t)-I_0|\MP\varepsilon^{2\gamma}, \quad 0 \leq t<+\infty. \] 

In the second case, there exists a smallest time $0\leq t^*<T$ such that $\omega(t^*)$ belongs to the set $R_K$. We apply once again lemma~\ref{lemmenonresonant} with $\tau=t^*$ to have 
\[ |I(t)-I_0|\MP \varepsilon^{2\gamma}, \quad 0\leq t\leq t^*. \]
Again, taking $\varepsilon_0$ small enough we can ensure that $I(t^*)\in B(0,R/2)$, then we can apply lemma~\ref{lemmePos} to the solution $I_{t^*}(t)=I(t+t^*)$, which initial frequency belongs to $R_K$, to obtain 
\[ |I_{t^*}(t)-I_{t^*}(0)| \MP \left(\varepsilon K^{2}\right)^{\frac{1}{2(n-1)}}, \quad 0\leq t \MP \exp\left(\cdot \varepsilon K^{2}\right)^{-\frac{1}{2(n-1)}}. \]
Setting
\[ a_\gamma=\frac{1-2\gamma}{2(n-1)} \]
this gives
\[ |I_{t^*}(t)-I_{t^*}(0)| \MP \varepsilon^{a_\gamma}, \quad 0 \leq t \MP \exp(\cdot \varepsilon^{-a_\gamma}). \]
Since $t^*\geq 0$, we get in particular 
\[ |I(t)-I(t^*)| \MP \varepsilon^{a_\gamma}, \quad t^* \leq t \MP \exp(\cdot \varepsilon^{-a_\gamma}), \]
and we conclude that
\[ |I(t)-I_0| \MP (\varepsilon^{2\gamma}+\varepsilon^{a_\gamma}) \MP\max\left\{\varepsilon^{2\gamma},\varepsilon^{a_\gamma}\right\}, \quad 0\leq t \MP \exp(\cdot \varepsilon^{-a_\gamma}). \]
Now, using the condition $0<\gamma \leq (2n)^{-1}$ one has
\[ (2n)^{-1} \leq a_\gamma < (2(n-1))^{-1} \]
in particular
\[ \gamma \leq a_\gamma \]
and hence
\[  \max\left\{\varepsilon^{2\gamma},\varepsilon^{a_\gamma}\right\}\leq \varepsilon^{\gamma}.  \]
Therefore for all solutions, the estimates
\[ |I(t)-I_0| \MP \varepsilon^{\gamma}, \quad 0\leq t \MP \exp(\cdot \varepsilon^{-a_\gamma}), \]
hold true. 

Finally, to obtain our statement just set
\[ \delta=\gamma(n-1)^{-1} \]
so that
\[ a_\gamma=(2(n-1))^{-1}-\delta \]
hence
\[ |I(t)-I_0| \MP \varepsilon^{\delta(n-1)}, \quad 0 \leq t \MP \exp\left(\cdot \varepsilon^{-(2(n-1))^{-1}+\delta}\right) \]
all this provided that
\[ 0<\delta \leq (2n(n-1))^{-1}.\]
This ends the proof of the first part of the statement. For the second part, choosing $K$ in terms of $\rho$ but independent of $\varepsilon$, the proof is similar and even simpler, so we do not repeat the details. 
\end{proof}

\section{The Gevrey case}

In this section, we will prove Theorem~\ref{exposantNekG}. In fact, it will be enough to have a version of lemma~\ref{lemmePos} in the Gevrey case, as the geometric considerations of the previous section still apply with no changes. 

Here we shall use a result from~\cite{MS02}, which follows the method introduced of Lochak (\cite{Loc92}) from which results of improved stability near resonances actually originates. In the latter approach, the notion of ``order" of a resonance is more intrinsic, however it is also more difficult to compute. 
 
Let $\Lambda$ be a sub-module of $\Z^n$ of rank $r$, with $r\in\{1,\dots,n\}$, and choose  a basis $\{k^1, \dots,k^r\} \in \Z^n$ for $\Lambda$. We define the matrix $L$ of size $r\times n$ with integer entries whose rows are given by the vectors $k^i=(k_1^i, \dots, k_n^i)$, $1\leq i\leq r$, that is
\[
L=\begin{pmatrix}
k_1^1 & \cdots & k^1_n \\
\vdots & & \vdots \\         
k^r_1 & \cdots & k^r_n
\end{pmatrix} \in M_{r,n}(\Z).
\] 
Then it is an elementary result of linear algebra that there exists integers $d_1, \dots, d_r \in \Z$, satisfying the divisibility conditions $d_1|\dots|d_r$, such that $L$ is equivalent to the diagonal matrix
\[
\Delta=\begin{pmatrix}
d_1 & & 0 & 0 & \cdots & 0  \\
 & \ddots & & & \vdots  \\         
0 &  & d_r & 0 & \cdots & 0 
\end{pmatrix} \in M_{r,n}(\Z).
\]
Therefore one can write 
\begin{equation} \label{clambda}
L=B\Delta A
\end{equation}
for some matrices $A\in GL(n,\Z)$ and $B\in GL(r,\Z)$. 

The numbers $d_i$ are called the invariant factors of the module, and for a maximal module one can show that these numbers are all equal to one. The above normal form result can be proved equivalently by elementary operations on rows and columns or using the structure of finitely generated modules over a principal domain. 

One can easily check that $^t\!A$ sends the  standard sub-module (which is the one generated by the first $r$ vectors of the canonical basis of $\Z^n$) to the sub-module $\Lambda$. So quantitative informations about the sub-module are encoded in those matrices $A\in GL(n,\Z)$.

Following Lochak, we define $c_\Lambda$ (resp. $c_\Lambda'$) as the minimal value of the norm $|A^{-1}|$ (resp. of $|A|$) among all matrices $A\in GL(n,\Z)$ satisfying the relation~(\ref{clambda}) (it is easy to see that those constants depend only on $\Lambda$ and not on the choice of such a matrix). In the space $M_n(\Z)$ we may choose the norm $|\,.\,|$ induced by the usual supremum norm for vectors, which is nothing but the maximum of the sums of the absolute values of the elements in each row.

With those definitions, one can state the following stability result  in the Gevrey class.

\begin{theoreme}[Marco-Sauzin]\label{theoremSau}
Let $\Lambda$ be a $K$-sub-module of $\Z^n$ of rank $r$, with $r\in\{0,\dots,n-1\}$.
Assume that $\varepsilon\geq 0$ satisfies
\begin{equation*}
\varepsilon c_\Lambda^{5(n-r)} \MP 1, \quad \varepsilon c_\Lambda^{3(n-r)} c_\Lambda '^{2(n-r)} \MP 1,
\end{equation*}
and $H$ satisfies~(\ref{HamexpoG}). Then for any solution $(\theta(t),I(t))$ with $I_0\in B(0,R/2)$ and $d(\omega(0),R_\Lambda) \MP \sqrt \varepsilon$ the following estimates
\[ |I(t)-I_0| \MP c_\Lambda^{3/2}c_\Lambda' \varepsilon^{\frac{1}{2(n-r)}}, \quad t \MP \exp\left(\cdot c_\Lambda^{5(n-r)} \varepsilon \right)^{-\frac{1}{2\alpha(n-r)}}\]
hold true.
\end{theoreme}  

This is exactly the addendum to Theorem A in \cite{MS02}, to which we refer for a possible choice of stable constants. Note also that in \cite{MS02} the names of these constants are a bit different: there $c_\Lambda''$ stands for what we have called $c_\Lambda$, and another constant called $c_\Lambda$ is introduced which is obviously equivalent to ours.

It will be sufficient to restrict ourselves to the case where the sub-module is of rank one. Even in this case,
the constants $c_\Lambda$ and $c_\Lambda'$ are in general very difficult to compute, so we will use only the obvious estimate 
\[ c_\Lambda \leq |A^{-1}|, \quad c_\Lambda' \leq |A|, \]
for some suitable matrix $A$ satisfying~(\ref{clambda}) for a given $K$-sub-module $\Lambda$.

\begin{proposition}\label{propEucl2}
Let $\Lambda$ be a maximal $K$-sub-module of rank $1$. Then we have the estimate
\[c_\Lambda \leq n!K^{n-1}, \quad c_\Lambda' \leq K.  \]
\end{proposition}

Let us point out that these estimates are very rough, however it seems difficult to improve the exponent in $K$.

\begin{proof}
Let $k=(k_1,\dots,k_n)\in\Z^n\setminus\{0\}$ be a generating vector of $\Lambda$, with $|k|\leq K$. Since $\Lambda$ is maximal,  the components of $k$ are relatively prime and the invariant factor of $\Lambda$ is equal to one, so
\[
B=(1),\qquad
\Lambda=(1\ 0\ldots 0),
\]
and a matrix $A\in GL(n,\Z)$ satisfies $L=B\Delta A$ as in~(\ref{clambda}) if and only if its first row is equal to $k$. 

\paraga To prove our estimates, it will be enough to show that one can choose $A$ such that the $\ell^1$--norm of each of its rows is bounded by $K$, that is $|A|\leq K$. Indeed, assuming the existence of such a matrix $A$, one immediately gets
\[ c_\Lambda' \leq |A| \leq K. \]
Moreover, since the determinant of $A$ is $\pm 1$, $A^{-1}$ is a matrix of cofactors, therefore the absolute value of each of its element is trivially bounded by $(n-1)!K^{n-1}$, which gives
\[ c_\Lambda \leq |A^{-1}| \leq n!K^{n-1}. \]
So it remains to prove that one can construct such a matrix, and we will do this by induction on $n\geq 1$.

\paraga Let us state first an elementary remark. 
Let $x\in \Z^2 \setminus \{0\}$ and $y\in \Z^2 \setminus \{0\}$, with $d={\rm gcd}(x,y)$. Then there exist $u,v\in \Z$ satisfying the equation
\[ ux+vy=d \]
with the estimates 
\[ |u|\leq \frac{|y|}{d}, \quad |v|\leq \frac{|x|}{d}.\] 

To see this, note that the existence of at least one solution $u_0,v_0$ follows easily from euclidean division algorithm. Then obviously
\[ u=u_0-k\frac{y}{d}, \quad v=v_0-k\frac{x}{d}, \]
is also a solution for any $k\in\Z$. Therefore choosing $k$ properly we can find at least one solution $u,v$ with
\[ |u|\leq \frac{|y|}{d}-1.  \]
For this solution one has
\begin{eqnarray*}
|v||y| & = & |d-ux| \\
& \leq & d+|ux| \\
& \leq & d+(\pfrac{|y|}{d}-1)|x| \\
& \leq & \pfrac{|y|}{d}|x|+d-|x|\\
& \leq & \pfrac{|y|}{d}|x|
\end{eqnarray*}
since $d-|x|\leq 0$, so  
\[|v|\leq \frac{|x|}{d}\] 
which proves our claim. Finally, if for instance $y=0$, then $d=x$ and one can obviously choose $u=1$ and $v=0$.

\paraga Let $K\geq 1$ be given. We can now state our induction hypothesis {\bf H(n)} for  $n\geq 1$. 

\vskip2mm

{\bf H(n).} {\em Let $k=(k_1,\ldots, k_n)$ be vector of $\Z^n\setminus \{0\}$ with co-prime components such that $|k|\leq K$. Then there exists a matrix $A\in GL(n,\Z)$ with first row equal to $k$, which satisfies $|A|\leq K$.}

\vskip2mm

The assertion {\bf H(1)} is immediate since in this case $k=(\pm 1)$. 
Now for $n\geq 2$, assume that  {\bf H(n-1)} holds true and consider $k=(k_1,\ldots, k_n)$ in $\Z^n\setminus \{0\}$ with co-prime components and $|k|\leq K$.

We may suppose that $k_*=(k_1,\dots,k_{n-1})$ is non-zero (otherwise we consider $k_*=(k_2, \dots, k_n)$) and we set $d={\rm gcd}(k_1,\dots,k_{n-1})$. So $d\geq1$, the integers $d^{-1}k_1,\dots,d^{-1}k_{n-1}$ are co-prime and
\[|k_*|\leq \frac{K}{d}.\]

By {\bf H(n-1)}  we can find a matrix 
\[
\begin{pmatrix}
d^{-1}k_1 & \cdots & d^{-1}k_{n-1} \\
l_{2,1} & \cdots & l_{2,n-1}  \\
\vdots & & \vdots \\         
l_{n-1,1} & \cdots & l_{n-1,n-1}
\end{pmatrix} \in GL(n-1,\Z),
\] 
such that 
\[\abs{\sum_{j=1}^{n-1} l_{i,j}} \leq K\]
 for each $i\in\{2,\dots,n-1\}$.
 
Now since $d$ and $k_n$ are co-prime, one can find integers $u,v\in \Z$ such that 
\[ ud+vk_n=1 \] 
and therefore define a matrix
\[
A(u,v)=\begin{pmatrix}
k_1 & \cdots & k_{n-1} & k_{n} \\
l_{2,1} & \cdots & l_{2,n-1} & 0  \\
\vdots & & \vdots & \vdots \\         
l_{n-1,1} & \cdots & l_{n-1,n-1} & 0 \\
(-1)^{n-1}vd^{-1}k_1 & \cdots & (-1)^{n-1}vd^{-1}k_{n-1} & (-1)^{n-1}u
\end{pmatrix}.
\] 
Expanding the determinant relatively to the last column easily proves that $A(u,v)\in GL(n,\Z)$. 

As for the estimates, first assume that $k_n=0$. Then $d=1$ and we may choose $u=1$ and $v=0$, so obviously
\[\abs{A(1,0)}\leq K.\]
If now $k_n$ is non-zero, then by our previous remark  we can choose $(u_*,v_*)$ so that
\[ |u_*|\leq |k_n|, \quad |v_*|\leq d,\]
which proves that the $\ell^1$--norm of the last row is bounded by $K$, and therefore  $\abs{A(u_*,v_*)}\leq K$.
This ends the proof.  
\end{proof}

With these estimates, one can deduce from Theorem~\ref{theoremSau} the following lemma.

\begin{lemme}\label{lemmeSau}
Let $\Lambda$ be a $K$-sub-module of $\Z^n$ of rank $1$. Assume that $\eps\geq 0$ and $K\geq 1$ satisfy
\begin{equation} \label{smalln1'}
\varepsilon K^{5(n-1)^2} \MP 1
\end{equation}
and $H$ satisfies (\ref{HamexpoG}). Then for any solution $(\theta(t),I(t))$ such that $I_0 \in B(0,R/2)$ and $\omega(0)\in R_K$ the following estimates
\[ |I(t)-I_0| \MP (\varepsilon K^{(n-1)(3n-1)})^{\frac{1}{2(n-1)}}, \quad t \MP \exp\left(\cdot \varepsilon K^{5(n-1)^2}\right)^{-\frac{1}{2\alpha(n-1)}}, \]
hold true.
\end{lemme}  

The proof of Theorem~\ref{exposantNekG} is now completely similar to the proof of Theorem~\ref{exposantNek}, but we shall repeat some details.

\begin{proof}[Proof of Theorem~\ref{exposantNekG}]
We choose $K$ of the form
\[ K=K_0\left(\frac{\varepsilon_0}{\varepsilon}\right)^{\gamma} \]
with suitable stable constants $K_0$ and $\varepsilon_0$ so that conditions~(\ref{smalln2}) and~(\ref{smalln1'}) are satisfied if
\[ 0\leq\varepsilon\leq\varepsilon_0, \quad 0<\gamma \leq 5^{-1}(n-1)^{-2}. \]

Then we can apply both lemma~\ref{lemmePos} and lemma~\ref{lemmenonresonant} in the same way as in the proof of Theorem~\ref{exposantNekG}, and setting
\[ a_\gamma=\frac{1-5\gamma(n-1)^{2}}{2\alpha(n-1)} ,\quad b_\gamma=\frac{1-\gamma(n-1)(3n-1)}{2(n-1)}, \]
we find that all solutions $(\theta(t),I(t))$ starting at $(\theta_0,I_0)$, with $I_0\in B(0,R/2)$, satisfy 
\[ |I(t)-I_0| \MP \max\left\{\varepsilon^{2\gamma},\varepsilon^{b_\gamma}\right\}, \quad t \MP \exp(\cdot \varepsilon^{-a_\gamma}). \]
Next using our condition
\[ 0 < \gamma \leq 5^{-1}(n-1)^{-2} \]
we have the bounds
\[ 5^{-1}(n-2)(n-1)^{-2} \leq b_\gamma \leq (2(n-1))^{-1} \]
hence $\gamma \leq b_\gamma$ so that
\[ |I(t)-I_0| \MP \varepsilon^\gamma, \quad t \MP \exp(\cdot \varepsilon^{-a_\gamma}). \]
To conclude, just set
\[ \delta=\frac{5}{2}\gamma(n-1) \]
so that
\[ a_\gamma=(2\alpha(n-1))^{-1}-\delta \]
hence
\[ |I(t)-I_0| \MP \varepsilon^{\frac{2}{5}\delta(n-1)^{-1}}, \quad |t| \MP \exp\left(\cdot \varepsilon^{-(2\alpha(n-1))^{-1}+\delta}\right) \]
all this provided that
\[ 0<\delta \leq (2(n-1))^{-1}\]
from which we will only retain
\[ 0<\delta \leq (2\alpha n(n-1))^{-1}.\]
This concludes the proof of the first part, and the arguments for the second part are analogous choosing $K$ that depends on $\rho$ but independent of $\varepsilon$.   
\end{proof}

\noindent{\bf Acknowledgments.} 
The authors thank Pierre Lochak for numerous discussions on the problem of optimality of the stability exponents, and A.B. also thanks Viet Loc Bui for his explanations on some arithmetic issues and Laurent Niederman for his support. 

\addcontentsline{toc}{section}{References}
\bibliographystyle{amsalpha}
\bibliography{ExpOpt5}

\providecommand{\bysame}{\leavevmode\hbox to3em{\hrulefill}\thinspace}
\providecommand{\MR}{\relax\ifhmode\unskip\space\fi MR }
\providecommand{\MRhref}[2]{%
  \href{http://www.ams.org/mathscinet-getitem?mr=#1}{#2}
}
\providecommand{\href}[2]{#2}
\begin{thebibliography}{LNN94}

\bibitem[Arn64]{Arn64}
V.I. Arnold, \emph{Instability of dynamical systems with several degrees of
  freedom}, Sov. Math. Doklady \textbf{5} (1964), 581--585.

\bibitem[Bes96]{Bes96}
U.~Bessi, \emph{{An approach to {A}rnold's diffusion through the calculus of
  variations}}, Nonlinear Anal., Theory Methods Appl. \textbf{26} (1996),
  no.~6, 1115--1135.

\bibitem[Bes97]{Bes97}
\bysame, \emph{{Arnold's example with three rotators}}, Nonlinearity
  \textbf{10} (1997), no.~3, 763--781.

\bibitem[Kol54]{Kol54}
A.N. Kolmogorov, \emph{On the preservation of conditionally periodic motions
  for a small change in {H}amilton's function}, Dokl. Akad. Nauk. SSSR
  \textbf{98} (1954), 527--530.

\bibitem[LM05]{LM05}
P~Lochak and J.P. Marco, \emph{Diffusion times and stability exponents for
  nearly integrable analytic systems}, Central European Journal of Mathematics
  \textbf{3} (2005), no.~3, 342--397.

\bibitem[LN92]{LN92}
P.~Lochak and A.I. Neishtadt, \emph{Estimates of stability time for nearly
  integrable systems with a quasiconvex {H}amiltonian}, Chaos \textbf{2}
  (1992), no.~4, 495--499.

\bibitem[LNN94]{LNN94}
P.~Lochak, A.I. Neistadt, and L.~Niederman, \emph{{Stability of nearly
  integrable convex {H}amiltonian systems over exponentially long times}},
  {Kuksin, S. (ed.) et al., Seminar on dynamical systems. Basel: Birkh\"auser.
  Prog. Nonlinear Differ. Equ. Appl. 12, 15-34 (1994).}, 1994.

\bibitem[Loc92]{Loc92}
P.~Lochak, \emph{{Canonical perturbation theory via simultaneous
  approximation}}, Russ. Math. Surv. \textbf{47} (1992), no.~6, 57--133.

\bibitem[MS02]{MS02}
J.-P. Marco and D.~Sauzin, \emph{Stability and instability for {G}evrey
  quasi-convex near-integrable {H}amiltonian systems}, Publ. Math. Inst. Hautes
  Études Sci. \textbf{96} (2002), 199--275.

\bibitem[Nek77]{Nek77}
N.N. Nekhoroshev, \emph{An exponential estimate of the time of stability of
  nearly integrable {H}amiltonian systems}, Russian Math. Surveys \textbf{32}
  (1977), no.~6, 1--65.

\bibitem[Nek79]{Nek79}
\bysame, \emph{An exponential estimate of the time of stability of nearly
  integrable {H}amiltonian systems {II}}, Trudy Sem. Petrovs \textbf{5} (1979),
  5--50.

\bibitem[Nie04]{Nie04}
L.~Niederman, \emph{Exponential stability for small perturbations of steep
  integrable {H}amiltonian systems}, Erg. Th. Dyn. Sys. \textbf{24} (2004),
  no.~2, 593--608.

\bibitem[Pös93]{Pos93}
J.~Pöschel, \emph{Nekhoroshev estimates for quasi-convex {H}amiltonian
  systems}, Math. Z. \textbf{213} (1993), 187--216.

\bibitem[Pös01]{Pos01}
\bysame, \emph{A lecture on the classical {KAM} theory}, Katok, Anatole (ed.)
  et al., Smooth ergodic theory and its applications (Seattle, WA, 1999).
  Providence, RI: Amer. Math. Soc. (AMS). Proc. Symp. Pure Math. 69, 707-732,
  2001.

\bibitem[Sev06]{Sev06}
Mikhail~B. Sevryuk, \emph{{Partial preservation of frequencies in KAM theory}},
  Nonlinearity \textbf{19} (2006), no.~5, 1099--1140.

\bibitem[Zha09]{KZ09}
Ke~Zhang, \emph{Speed of {A}rnold diffusion for analytic {H}amiltonian
  systems}, Preprint (2009).

\end{thebibliography}
\end{document}